%% file: BL14.tex
\documentclass[11pt]{article}
\usepackage{amsmath}
\usepackage{amssymb,amsbsy}
\usepackage{graphicx}
\usepackage{booktabs}
\usepackage{dsfont}
\usepackage{natbib}
\usepackage[margin=1in]{geometry}
\usepackage{tikz}
\usetikzlibrary{arrows,decorations.pathmorphing,backgrounds,positioning,fit,petri,mindmap,shapes.geometric,decorations.pathreplacing}

\usepackage{pdfsync}
\usepackage{hyperref}

\input{Commands}

\newcommand{\BlackBox}{\rule{1.5ex}{1.5ex}}  
\newenvironment{proof}{\par\noindent{\bf Proof\ }}{\hfill\BlackBox\\[2mm]}

\begin{document}
\title{On the local profiles of trees}
\author{S{\'e}bastien Bubeck\footnote{Department of Operations Research and Financial Engineering, Princeton University, Princeton 08540, USA. Email: \texttt{sbubeck@princeton.edu}.} \and Nati Linial\footnote{School of Computer Science and Engineering, The Hebrew University of Jerusalem, Jerusalem 91904, Israel. Email: \texttt{nati@cs.huji.ac.il}.}\thanks{Research supported in part by grants from the ISF and I-Core.}}
\date{}

\maketitle

\abstract{We study the local profiles of trees. We show that, in contrast with the situation for general graphs, the limit set of $k$-profiles of trees is convex. We initiate a study of the defining inequalities of this convex set. Many challenging problems remain open.}

\section{Introduction}
For (unlabelled) trees $T$, $S$, we denote by $c(S, T)$ the number of copies of $S$ in $T$, or in other words the number of injective homomorphism from $S$ to $T$. Let $T_1^k, \hdots, T_{N_k}^k$ be a list of all (isomorphism types of) $k$-vertex trees\footnote{Recall that the sequence $(N_k)_{k \geq 1}$ starts with $1,1,1,2,3,6\ldots$.}, where $T_1^k, T_2^k$ are the $k$-vertex path and the $k$-vertex star, respectively. {\em The $k$-profile of a tree $T$} is the vector $p^{(k)}(T) \in \R^{N_k}$ whose $i$-th coordinate is
$$(p^{(k)}(T))_i = \frac{c(T^k_i, T)}{Z_k(T)}, \ \text{where} \ Z_k(T) = \sum_{j=1}^{N_k} c(T^k_j, T).$$
In other words the $k$-profile is the induced density vector of $k$-vertex trees.
We are interested in understanding the limit set of $k$-profiles:
$$\Delta_{\cT}(k) = \left\{p \in \R^{N_k} : \exists (T_n), |T_n| \xrightarrow[n \rightarrow \infty]{} \infty, \, \text{and} \, p^{(k)}(T_n) \xrightarrow[n \rightarrow \infty]{} p \right\} ,$$
where $|T|$ denotes the number of vertices in $T$. 
Our main result, proved in Section \ref{sec:conv}, is:
\begin{theorem} \label{th:conv}
The set $\Delta_{\cT}(k)$ is convex.
\end{theorem}
This property of profiles of trees is in sharp contrast with what happens for general graphs. Let $\Delta(k)$ be the $k$-profiles limit set of general graphs (which is defined like $\Delta_{\cT}(k)$ with a list of all $k$-vertex graphs rather than $k$-vertex trees). The first and second coordinates in $p \in \Delta(k)$ correspond to $k$-anticliques and $k$-cliques respectively. Clearly $e_1 = (1,0,\hdots,0), e_2=(0,1,0,\hdots,0) \in \Delta(k)$ but $\frac12 e_1 + \frac12 e_2 \not\in \Delta(k)$. Not only is $\Delta(k)$ nonconvex, it is even computationally infeasible to derive a description of its convex hull, see \cite{HN11}. Our understanding of the sets $\Delta(k)$ is rather fragmentary (e.g. \cite{HLNPS12}). Flag algebras \cite{Raz07} are a major tool in such investigations. The convexity of $\Delta_{\cT}(k)$ suggests that we may have a better chance understanding profiles of trees, by deriving the linear inequalities that define these sets. We take some steps in this direction. Concretely we prove the following result in Section \ref{sec:sp}.
\begin{theorem} \label{th:12k}
Let $p \in \Delta_{\cT}(k)$, then
$$p_1 + p_2 \geq \frac{1}{2 N_k k^{2k}} .$$
\end{theorem}
We suspect that a stronger lower bound holds here. In Section \ref{sec:sp} we give examples which show that $p_1 +  p_2$ can be exponentially small in $k$. 

For $5$-profiles we get a better inequality. In Section \ref{sec:5profile} we prove
\begin{theorem} \label{th:125}
Let $p \in \Delta_{\cT}(5)$, then
$$p_2 \geq \frac{1-2p_1}{37}.$$
\end{theorem}
The above inequality holds with equality at the point $(1/2, 0) \in \Delta_{\cT}(5)$, but we believe that it is not tight for $p \in \Delta_{\cT}(5)$ such that $p_2 > 0$. We discuss tightness in more detail in Section \ref{sec:5profile}.
We end the paper with a list of open problems in Section \ref{sec:open}.

\section{Convexity of the $k$-profiles limit set} \label{sec:conv}
In this Section we prove Theorem \ref{th:conv}. We first explain how to ``glue" two trees, and then we show how gluing allows us to generate convex combinations of tree profiles.
\newline

\noindent
\textbf{Step 1: the gluing operation.}
If $T$ and $S$ are trees, we define $T \boxtimes_k S$ as follows. This is a tree which consists of a copy of $T$, a copy of $S$ and a $(k-1)$-vertex path that connects some arbitrary leaf $x$ in $T$ to an arbitrary leaf $y$ in $S$. In other words, we add to $S$ and $T$ a path $x=z_0,\ldots,z_k=y$ where $z_1,\ldots,z_{k-1}$ are new vertices. The resulting tree depends of course on the choice of the two leaves $x$ and $y$, but we ignore this issue, since this will not affect anything that is said below.

We denote by $D(K)$ the largest vertex degree in a given tree $K$. The following inequalities are easy to verify:
\begin{align}
c(T^k_i, T) + c(T^k_i, S) \leq c(T^k_i, T \boxtimes_k S) \leq c(T^k_i, T) + c(T^k_i, S) + k D(T)^{k-2} + k D(S)^{k-2} , \label{eq:conv2}
\end{align}
and consequently
\begin{align}
& Z_k(T) + Z_k(S) \leq Z_k(T \boxtimes_k S) \leq Z_k(T) + Z_k(S) + k N_k D(T)^{k-2} + k N_k D(S)^{k-2} . \label{eq:conv1}
\end{align}
We define by induction $T^{\boxtimes_k \ell} = T^{\boxtimes_k (\ell -1)} \boxtimes_k T$ (with $T^{\boxtimes_k 1} = T$). Observe that $D(T^{\boxtimes_k \ell}) = D(T)$ and thus using \eqref{eq:conv2} and \eqref{eq:conv1} one has 
\begin{align}
& \ell c(T^k_i, T) \leq c(T^k_i, T^{\boxtimes_k \ell}) \leq \ell c(T^k_i, T) + 2k (\ell-1) D(T)^{k-2} , \label{eq:conv4}\\
& \ell Z_k(T) \leq Z_k(T^{\boxtimes_k \ell}) \leq \ell Z_k(T) + 2 k N_k (\ell-1) D(T)^{k-2} . \label{eq:conv3}
\end{align}
\newline

\noindent
\textbf{Step 2: convex combinations by gluing.} Let $p, q \in \Delta_{\cT}(k)$. Namely, there exists two sequences of trees $T_n$ and $S_n$ such that 
$$|T_n|, |S_n| \xrightarrow[n \rightarrow \infty]{} \infty, \, \text{and} \, (p^{(k)}(T_n), p^{(k)}(S_n)) \xrightarrow[n \rightarrow \infty]{} (p,q) .$$
Now, given $\lambda \in (0,1)$, we want to construct a sequence of trees $R_n$ such that 
$$|R_n| \xrightarrow[n \rightarrow \infty]{} \infty, \, \text{and} \, p^{(k)}(R_n) \xrightarrow[n \rightarrow \infty]{} \lambda p + (1-\lambda) q .$$
First let $\alpha_n/\beta_n$ be a sequence of rational numbers that converges to $\lambda$. We correspondingly define the sequence of trees $R_n$ via:
$$R_n = T_n^{\boxtimes_k [\alpha_n Z_k(S_n)]} \boxtimes_k S_n^{\boxtimes_k [(\beta_n - \alpha_n) Z_k(T_n)]} .$$
Using \eqref{eq:conv1} and \eqref{eq:conv3} one immediately obtains
\begin{align}
& \beta_n Z_k(T_n) Z_k(S_n) \notag \\
& \leq Z_k(R_n) \label{eq:conv} \\
& \leq \beta_n Z_k(T_n)Z_k(S_n) + 2 k N_k \alpha_n Z_k(S_n) D(T_n)^{k-2} + 2 k N_k (\beta_n - \alpha_n) Z_k(T_n) D(S_n)^{k-2} . \notag
\end{align}
Now the key observation is that
\begin{equation} \label{eq:conv5}
D(T_n)^{k-2} = o(Z_k(T_n)) .
\end{equation}
Indeed, $Z_k(T_n) \geq {D(T_n) \choose k-1}$ follows by counting $k$-vertex stars rooted at the highest degree vertex in $T_n$, which yields equation \eqref{eq:conv5} if $D(T_n) \to \infty$. On the other hand, if $D(T_n)$ is bounded then \eqref{eq:conv5} is also clearly true since $Z_k(T_n) \to \infty$.

Using \eqref{eq:conv5} one can rewrite \eqref{eq:conv} as
$$Z_k(R_n) = \beta_n Z_k(T_n) Z_k(S_n) + o(\beta_n Z_k(T_n) Z_k(S_n)) .$$
Similarly using \eqref{eq:conv2} and \eqref{eq:conv4} we obtain
$$c(T_i^k, R_n) = \alpha_n Z_k(S_n) c(T_i^k, T_n) + (\beta_n - \alpha_n) Z_k(T_n) c(T_i^k, S_n) + o(\beta_n Z_k(T_n) Z_k(S_n)) .$$
We combine these two identities and conclude that
$$( p^{(k)}(R_n) )_i = \frac{c(T_i^k, R_n)}{Z_k(R_n)} = (1+o(1)) \left[ \frac{\alpha_n}{\beta_n} \frac{c(T_i^k, T_n)}{Z_k(T_n)} + \left( 1 - \frac{\alpha_n}{\beta_n} \right) \frac{c(T_i^k, S_n)}{Z_k(S_n)} \right] + o(1) \to \lambda p_i + (1-\lambda) q_i ,$$
as claimed.

\section{On Stars and Paths} \label{sec:sp}
In this Section we prove Theorem \ref{th:12k}. We use the shorthand $P_k(T)=c(T^k_1, T)$ and $S_k(T)=c(T^k_2, T)$, and we also omit the reference to $T$ whenever it is clear from context. Before delving into the proof let us show why the exponential decrease in $k$ is unavoidable. 
A $d$-{\em millipede} is a tree where all non-leaf vertices reside on a single path and they have degree $d+2$ each. See Figure \ref{fig:dmilli} for an illustration. The number of non-leaves is called the $d$-millipede's {\em length}. 
\begin{figure}
\begin{center}
\begin{tikzpicture}[scale=2]
\draw (0,0) -- (1,0) -- (2,0);
\draw[dashed] (2.2,0) -- (3.3,0);
\draw (3.5,0) -- (4.5,0) -- (5.5,0);
\draw (1,0) -- (1.3,-1);
\draw (1,0) -- (0.7,-1);
\draw (2,0) -- (2.3,-1);
\draw (2,0) -- (1.7,-1);
\draw (3.5,0) -- (3.2,-1);
\draw (3.5,0) -- (3.8,-1);
\draw (4.5,0) -- (4.2,-1);
\draw (4.5,0) -- (4.8,-1);
\fill (0,0) circle (0.05);
\fill (1,0) circle (0.05);
\fill (2,0) circle (0.05);
\fill (3.5,0) circle (0.05);
\fill (5.5,0) circle (0.05);
\fill (4.5,0) circle (0.05);
\fill (3.2,-1) circle (0.05);
\fill (3.8,-1) circle (0.05);
\fill (4.2,-1) circle (0.05);
\fill (4.8,-1) circle (0.05);
\fill (0.7,-1) circle (0.05);
\fill (1.3,-1) circle (0.05);
\fill (2.3,-1) circle (0.05);
\fill (1.7,-1) circle (0.05);
\draw[dashed] (0.8,-1) -- (1.2,-1);
\draw[dashed] (1.8,-1) -- (2.2,-1);
\draw[dashed] (3.3,-1) -- (3.7,-1);
\draw[dashed] (4.2,-1) -- (4.7,-1);
\draw[decorate, decoration={brace, mirror, amplitude=5pt}]
(0.5,-1.2) -- (1.5,-1.2) node[below, midway, yshift=-0.25cm] {$d$};
\end{tikzpicture}
\end{center}
\caption{A $d$-millipede.}
\label{fig:dmilli}
\end{figure}
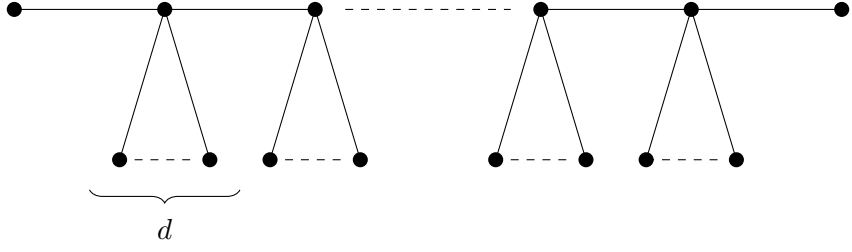
We denote by $T_n$ the $(k-4)$-millipede of length $n$ with $k$ even. Also, $R_k$ is the $\frac{k-4}{2}$-millipede of length $2$. It is easy to see that for $k \geq 6$,
$$Z_{k}(T_n) \geq c(R_k, T_n) \geq 2 (n-2) {k-3 \choose (k-2)/2} \geq (n-2) (3/2)^{k/2} ,$$
and
$$S_k(T_n) = 0, \qquad P_k(T_n) \leq n (k-3)^2 .$$
Thus the limiting profile that corresponds to the sequence $(T_n)$ satisfies
\begin{equation} \label{eq:addedafterreview}
p_1 + p_2 \leq \frac{(k-3)^2}{(3/2)^{k/2}} .
\end{equation}
For $k>3$, let $\cP(k)$ be the projection of $\Delta(k)$ on the first two coordinates, that is
$$\cP(k) = \{(p_1,p_2), p \in \Delta_{\cT}(k)\} .$$
As a side note we also observe that the above inequality yields:
$$\overline{{\cup_k \cP(k)}}= \{x \in \R_+^2 :  x_1 + x_2 \leq 1\} ,$$
where $\overline{A}$ denotes the closure of a set $A$. Indeed $(1,0)$ and $(0,1)$ are always in $\cP(k)$, \eqref{eq:addedafterreview} shows that for $k$ large enough one can find a point arbitrarily close to $(0,0)$, and thus using the convexity of $\cP(k)$ (Theorem \ref{th:conv}) one obtains the above set equality.

We now turn to the proof of Theorem \ref{th:12k}. We repeatedly use the following obvious result which we state without a proof.
\begin{lemma} \label{lem:obv}
A tree with maximal degree $D$ has at most $k N_k D^{k-1}$ $k$-vertex subtrees that contain a given vertex.
\end{lemma}

Lemma \ref{lem:12k} is an enumerative analog of the probabilistic statement of Theorem \ref{th:12k} which applies when $S_k=0$. In Lemma \ref{lem:12k2} we deal with the case of $S_k \geq 0$, which then yields Theorem \ref{th:12k}.
\begin{lemma} \label{lem:12k}
If $D(T)\le k-2$ for some tree $T$, then
$$Z_k \leq k N_k (k-2)^{k-1} P_k + k N_k (k-2)^{2k-2} .$$
\end{lemma}

\begin{proof}
For trees with $n\le (k-2)^{k-1}$ vertices this inequality clearly follows from Lemma \ref{lem:obv}. For $n > (k-2)^{k-1}$ we proceed by induction. Clearly for this range of $n$, the tree's diameter must be at least $2(k-2)$. In other words it must contain a copy $P$ of $P_{2(k-2)+1}$. Let the tree $T'$ be obtained by removing a leaf $x$ from $T$. This eliminates at least one $k$-vertex path, namely the path from $x$ toward $P$ possibly proceeding toward $P$'s furthest end. In other words:
$$P_k(T) \geq P_k(T')+1 .$$ Furthermore by Lemma \ref{lem:obv}
$$Z_k(T) \leq Z_k(T') + k N_k(k-2)^{k-1} .$$
Applying the induction hypothesis to $T'$ yields
$$Z_k(T') \leq k N_k (k-2)^{k-1} P_k(T') + k N_k (k-2)^{2k-2} ,$$
together with the two above inequalities this gives the same inequality for $T$.
\end{proof}

\begin{lemma} \label{lem:12k2}
Every tree satisfies
$$Z_k \leq N_k k^{2k} (P_k+2S_k+1) .$$
\end{lemma}

\begin{proof}
First observe that if $n \leq k^k$ then by (a variant of) Lemma \ref{lem:obv}:
\begin{eqnarray*}
Z_k & \leq & \sum_{u : d(u) \leq k-2} k N_k d(u)^{k-1} + \sum_{u : d(u) \geq k-1} k N_k d(u)^{k-1} \\
& \leq & N_k k^{2k} + \sum_{u : d(u) \geq k-1} k N_k (k-1)^{k-1} {d(u) \choose k-1} \\
& \leq & N_k k^{2 k} + N_k k^k  S_k,
\end{eqnarray*}
as needed. For larger trees we prove the following stronger inequality by induction on the number of vertices:
$$Z_k \leq N_k k^{2k} \left(P_k + 1\right) \ds1\{P_k \geq 1\} +  2 N_k k^{2k} S_k.$$
Clearly the expression $\ds1\{P_k \geq 1\}$ captures the information whether or not $T$'s diameter is at least $k-1$. The base case $n=k^k$ follows since necessarily $P_k \geq 1$ or $S_k \geq 1$. The induction step has two cases:

\textbf{Case 1:} If $D(T)\le k-2$, then Lemma \ref{lem:12k} yields the inequality, since $P_k \geq 1$.

\textbf{Case 2:} Let $v$ be the vertex of largest degree $d \geq k-1$, and let $T_1, \hdots, T_d$ be the trees of the forest $T\setminus\{v\}$. By Lemma \ref{lem:obv}
$$Z_k(T) \leq \sum_{i=1}^d Z_k(T_i) + k N_k d^{k-1} .$$
Furthermore
$$S_k(T) \geq \sum_{i=1}^d S_k(T_i) + {d \choose k-1} \geq \sum_{i=1}^d S_k(T_i) + \left(\frac{d}{k-1}\right)^{k-1} ,$$
and 
$$(1+P_k(T)) \ds1\{P_k(T) \geq 1\} \geq \sum_{i=1}^d (1+P_k(T_i)) \ds1\{P_k(T_i) \geq 1\}.$$
To see why the last inequality holds true, observe first that it is trivial if $\sum_{i=1}^d \ds1\{P_k(T_i)\geq1\} \in \{0,1\}$. Furthermore if $\sum_{i=1}^d \ds1\{P_k(T_i)\geq1\} \geq 2$, then for each $i$ such that $P_k(T_i) \geq 1$, one can find a path in $T$ containing both $v$ and vertices from $T_i$, which means that in this case one even has $P_k(T) \geq \sum_{i=1}^d (1+P_k(T_i)) \ds1\{P_k(T_i) \geq 1\}$.

Combine the three above displays and apply induction to the $T_i$'s to conclude:
\begin{eqnarray*}
Z_k(T) & \leq & \sum_{i=1}^d Z_k(T_i) + k N_k d^{k-1} \\
& \leq & N_k k^{2k} \sum_{i=1}^d \left(P_k(T_i) + 1\right) \ds1\{P_k(T_i) \geq 1\} + 2 N_k k^{2k} \sum_{i=1}^d S_k(T_i) + k N_k d^{k-1}\\
& \leq & N_k k^{2k} (1+P_k(T)) \ds1\{P_k(T) \geq 1\} + 2 N_k k^{2k} S_k(T),
\end{eqnarray*}
which concludes the proof.
\end{proof}

\section{$5$-profiles} \label{sec:5profile}

Clearly $\Delta(5)$ is entirely determined by $\cP(5)$. In this Section we prove Theorem \ref{th:125} which improves Theorem \ref{th:12k} for $k=5$. 

Before we embark on the proof we show that millipedes generate a 'large' set of points in $\cP(5)$. To simplify notation, let $P(T)=c(T^5_1, T)$, $S(T)=c(T^5_2, T)$ and $Y(T)=c(T^5_3, T)$ (note that $T^5_3$ has the $Y$-shape). We also omit the dependency on $T$ whenever it is clear from context. For a $d$-millipede of length $n$ we get the following expressions:
\begin{eqnarray*}
S & = & n {d+2 \choose 4} , \\
P & = & (n-2) (d+1)^2 , \\
Y & = & 2 (n-2) {d+1 \choose 2} (d+1) + 2 {d+1 \choose 2} (d+1) = (n-1) (d+1)^2 d ,\\
S + Y + P & = & n {d+2 \choose 4} + (n-2) (d+1)^3 + (d+1)^2 d .
\end{eqnarray*}
In particular for fixed $d$ and $n \to \infty$, we get the following point in $\cP(5)$:
\begin{equation} \label{eq:md}
m_d = \left(\frac{(d+1)^2}{{d+2 \choose 4} + (d+1)^3} , \frac{{d+2 \choose 4}}{{d+2 \choose 4} + (d+1)^3}\right) .
\end{equation}
Thus by convexity we have
\begin{equation} \label{eq:inclusion}
\cP(5) \supseteq \conv(\{(0,1)\}\cup\{ m_d, d \geq 0\}) .
\end{equation}
We cannot rule out the possibility that this is, in fact an equality.
This inclusion and the inequality from Theorem \ref{th:125} are illustrated in Figure \ref{fig:pic}.
\newline

\begin{figure}
\begin{center}
\includegraphics[scale = 0.5]{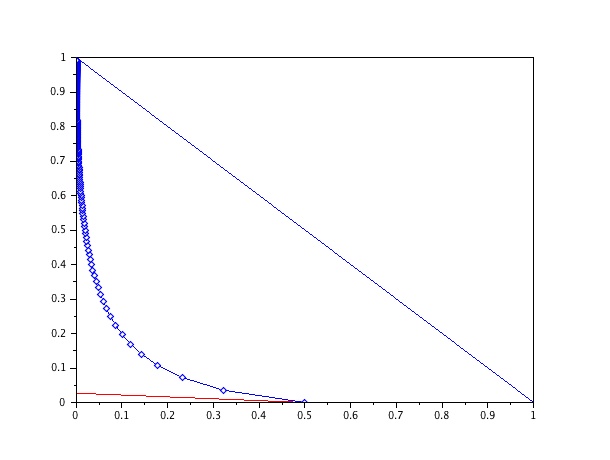}
\end{center}
\caption{The equation of the red line is $y=\frac{1-2x}{37}$. In blue: the polygonal curve connecting consecutive $m_d, d\geq1$ of equation \eqref{eq:md} as well as $(0,1)$ to $(1,0)$. By Theorem \ref{th:125} the set $\cP(5)$ lies above the red line and by Theorem \ref{th:conv} it contains the convex domain bounded by the blue lines.}
\label{fig:pic}
\end{figure}

Our proof of Theorem \ref{th:125} proceeds along the route that we took in proving Theorem \ref{th:12k}. Now, however, we are much more careful with the details. Lemma \ref{lem:1251}, a counterpart of Theorem \ref{th:125} gives an inequality on the unnormalized quantities when $S=0$. The general case $S \geq 0$ is handled in Lemma \ref{lem:1252} which yields Theorem \ref{th:125}.

\begin{lemma} \label{lem:1251}
If $D(T)\le 3$, then
$$Y \leq P + 4 ,$$
with equality if and only if $T$ is a $1$-millipede.
\end{lemma}
Note that to prove Theorem \ref{th:125} we will only need the inequality provided by Lemma \ref{lem:1251}.

\begin{proof}
It is immediate that a $1$-millipede satisfies $Y = P + 4$. We prove the inequality in two steps. A third step shows that only $1$-millipedes satisfy $Y = P + 4$.
\newline

\noindent
\textbf{Step 1: a formula for $P - Y$.}
We say that a vertex of degree $3$ has {\em type} $xyz$ with $x, y, z \in \{0,1,2\}$ if its three neighbors have degree $x+1, y+1$, and $z+1$, respectively. The number of vertices of type $xyz$ is denoted $n_{xyz}$. Similarly we define for degree-$2$ vertices the quantity $n_{xy}$.

A straightforward (but slightly painful) calculation yields
$$P = 12 n_{222} + 8 n_{221} + 4 n_{220} + 5 n_{211} + 2 n_{210} + 3 n_{111} + n_{110} + 4 n_{22} + 2 n_{21} + n_{11} ,$$
and
$$Y = 6 n_{222} + 5 n_{221} + 4 n_{220} + 4 n_{211} + 3 n_{210} + 2 n_{200} + 3 n_{111} + 2 n_{110} + n_{100} .$$
Hence
\begin{equation} \label{eq:1}
P - Y =  6 n_{222} + 3 n_{221} + n_{211} - n_{210} - 2 n_{200} - n_{110} - n_{100} + 4 n_{22} + 2 n_{21} + n_{11} .
\end{equation}

\noindent
\textbf{Step 2: double counting.} Let $n_x$ be the number of degree-$x$ vertices. Clearly $n_1 + n_2 + n_3 = n ,$
and by double counting of edges, also $n_1 + 2 n_2 + 3 n_3 = 2(n-1) .$
In particular,
\begin{equation} \label{eq:2}
n_1 - n_3 = 2.
\end{equation}
Next observe that $n_1$ and $n_3$ can easily be expressed in terms of the parameters $n_{xy}$ and $n_{xyz}$. Namely,
\begin{eqnarray*} 
n_3 & = & n_{222} + n_{221} + n_{220} + n_{211} + n_{210} + n_{200} + n_{111} + n_{110} + n_{100} , \\
n_1 & = & n_{220} + n_{210} + 2 n_{200} + n_{110} + 2 n_{100} + n_{20} + n_{10} . 
\end{eqnarray*}
Together with \eqref{eq:2} we find
\begin{equation} \label{eq:3}
- n_{222} - n_{221} - n_{211} + n_{200} - n_{111} + n_{100} + n_{20} + n_{10}  = 2.
\end{equation}
Next adding \eqref{eq:1} to twice \eqref{eq:3} one gets
$$P - Y + 4 =  4 n_{222} + n_{221} - n_{211} - n_{210} - 2 n_{111} - n_{110} + n_{100} + 4 n_{22} + 2 n_{21} + n_{11} + 2 n_{20} + 2 n_{10} .$$
It only remains to show that the right hand side term is non-negative. To this end we count  edges between a degree-$2$ vertex and a degree-$3$ vertex in two ways: Once from the degree-$3$ side and once from the degree-$2$ side
$$n_{221} + 2 n_{211} + n_{210} + 3 n_{111} + 2 n_{110} + n_{100} = 2 n_{22} + n_{21} + n_{20} .$$
This concludes the proof of the inequality stated in the theorem. Note that we have, in fact, showed a more precise statement:
\begin{equation} \label{eq:4}
P - Y + 4 =  4 n_{222} + 2 n_{221} + n_{211} + n_{111} + n_{110} + 2 n_{100} + 2 n_{22} + n_{21} + n_{11} + n_{20} + 2 n_{10} .
\end{equation}

\noindent
\textbf{Step 3: the equality case.} Equation \eqref{eq:4} shows that if $P - Y + 4=0$ then 
$$4 n_{222} + 2 n_{221} + n_{211} + n_{111} + n_{110} + 2 n_{100} + 2 n_{22} + n_{21} + n_{11} + n_{20} + 2 n_{10} = 0.$$
In particular the tree contains no degree-$2$ vertices, and no degree-$3$ vertices of type $222$. In other words, it has only leaves and degree-$3$ vertices of types $220$ and $200$. Moreover, by \eqref{eq:3} in this case $n_{200} = 2$. A straightforward inductive proof shows that the tree must be a $1$-millipede.
\end{proof}

We now adapt Lemma \ref{lem:1251} to the case where $S > 0$. This more general inequality directly implies Theorem \ref{th:125}.

\begin{lemma} \label{lem:1252}
All trees satisfy
$$Y \leq 36 S + P + 4.$$
\end{lemma}

\begin{proof}
First observe the following expressions 
$$Y = \sum_{\{u,v\} \in E} \left( {d(v) - 1 \choose 2} (d(u) -1) + {d(u) - 1 \choose 2} (d(v) -1) \right). $$
We split $Y= Y_s + Y_{\ell}$, where
$${Y_s} = \sum_{\{u,v\} \in E : \max(d(u), d(v)) \leq 3} \left( {d(v) - 1 \choose 2} (d(u) -1) + {d(u) - 1 \choose 2} (d(v) -1) \right), $$
and
$${Y_{\ell}} = \sum_{\{u,v\} \in E : \max(d(u), d(v)) \geq 4} \left( {d(v) - 1 \choose 2} (d(u) -1) + {d(u) - 1 \choose 2} (d(v) -1) \right). $$
The proof deals separately with ${Y_s}$ and ${Y_{\ell}}$.
\newline

\noindent
\textbf{Step 1:} We prove that ${Y_{\ell}} \leq 36 S$ by observing
$$S = \sum_{u \in V} {d(u) \choose 4} = \frac{1}{4} \sum_{u, v : \{u,v\} \in E} {d(u) - 1 \choose 3} = \frac{1}{4} \sum_{\{u,v\} \in E} \left({d(u) - 1 \choose 3} + {d(v) - 1 \choose 3}\right).$$
and making a term-by-term comparison with the expression for ${Y_{\ell}}$. We use the fact that for any nonnegative integers $x \neq 2, y \geq 3$
$$y x (x-1) + x y (y-1) \leq x^2(x-1) + y^2(y-1) \leq 3 (x(x-1)(x-2) + y(y-1)(y-2) ),$$
and furthermore for $x=2$ this inequality (without the intermediate step) is also true.
\newline

\noindent
\textbf{Step 2:}
We prove by induction on the size of the tree that ${Y_s} \leq P + 4$. The base case is trivial. The induction step has three cases:

\textbf{Case 1:} $D(T)\le 3$. The inequality follows readily from Lemma \ref{lem:1251}.

\textbf{Case 2:} There are two neighbors $u,v$ in $T$, where $d(u)\ge 4$ and $v$ is a leaf. Clearly,
$${Y_s}(T) \leq {Y_s}(T'), \ \text{and} \ P(T') \leq P(T)$$
where $T':=T\setminus\{v\}$.
By applying the induction hypothesis to $T'$ we see that ${Y_s}(T') \leq P(T') + 4$ which implies ${Y_s} \leq P + 4$.

\textbf{Case 3:} There is a vertex $u$ in $T$ with $d(u)\ge 4$, and no neighbor of $u$ is a leaf. Let $v$ be a neighbor of $u$ and let $T_1, T_2$ be the two trees of the forest obtained by removing the edge $uv$ and adding a new edge to $v$, where $u$ is in $T_1$ and $v$ in $T_2$. As in Case 2
$${Y_s}(T) \leq {Y_s}(T_1) + {Y_s}(T_2) .$$
Observe that we can assume that $v$ was selected such that $T_2$ has at least $3$ edges, for otherwise $Y_s(T)=0$ and thus the inequality would trivially hold. Indeed if $T_2$ had $2$ edges for all neighbors of $u$, then $T\setminus\{u\}$ would be a matching, and thus any copy of $T^5_3$ in $T$ would have $u$ in its ``middle edge'', which implies $Y_s(T)=0$.

Now clearly if $T_2$ has at least $3$ edges,
$$P(T) \geq P(T_1) + P(T_2) + 2 (d(u)-1) \geq P(T_1) + P(T_2) + 4.$$
Applying the induction hypothesis to $T_1$ and $T_2$ and using the above inequalities yield ${Y_s} \leq P + 4$ in this case as well.
\end{proof}

\section{Open problems} \label{sec:open}
\begin{enumerate}
\item Is the blue curve in Figure \ref{fig:pic} tight? That is, is \eqref{eq:inclusion} in fact an equality? Less ambitiously, can the bound in Lemma \ref{lem:1252} be improved to $Y \leq 9 S + P + K,$ for some universal $K \geq 0$ ? If true, this shows that the first segment of the polygonal curve is tight.
\item Recall that $\cP(k)$ is the projection of the limit set of $k$-profiles to the first two coordinates. Are these sets increasing, i.e., is it true that
$$\cP(k) \subset \cP(k+1)$$
for all integer $k$ ?
\item Let $p \in \Delta_{\cT}(k)$. Does $p_1=0$ imply $p_2 =1$?
\item Imitating a concept from graph theory we define the {\em inducibility} of a tree $T$ to be $\limsup \frac{c(T, {\cal T})}{Z_{|T|}({\cal T})}$ where the $\limsup$ is over trees ${\cal T}$ of size tending to infinity. By gluing many copies of $T$ as in Section \ref{sec:conv} it is easy to show that every $T$ has positive inducibility. 
By Theorem \ref{th:12k} paths and stars are the only trees with inducibility $1$, but are there other trees with inducibility arbitrarily close to $1$ ? If such trees do not exist, is it nonetheless possible to find infinitely many trees of inducibility $\ge \epsilon$ for some $\epsilon >0$ ? Note that in the realm of graphs there are infinitely many distinct graphs with inducibility $>\frac{1}{10}$, for example, the complete bipartite graphs $H=K_{3,r}$ with $r>10$. It can be easily verified that randomly chosen set of $r+3$ vertices in $K_{3n, rn}$ for $n$ large spans a copy of $H$ with probability $>0.1$.
\item Call a sequence of trees $(T_n)$ {\em $k$-universal} if
$$\liminf_{n \rightarrow \infty} (p^{(k)}(T_n))_i > 0$$ for every $i \in [N_k]$. The convexity of $\Delta_{\cT}(k)$ and the fact that every tree has positive inducibility implies that $k$-universal sequences exist. But does there exist a sequence of trees which is $k$-universal simultaneously for every $k$ ? For general graphs the answer is positive, e.g., using $G(n,p)$ graphs.
\item Is there a probabilistic interpretation to the profile of a tree?
\item In this paper we found only linear inequalities satisfied by the sets $\Delta_{\cal T}(k)$. We wonder if higher order inequalities can be derived as well. Is there a framework similar to flag algebras that applies to trees?
\end{enumerate}

\subsection*{Acknowledgements}
The research described here was carried out at the Simons Institute for the Theory of Computing. We are grateful to the Simons Institute for offering us such a wonderful research environment. We also thank an anonymous referee for fixing a mistake in the first version of this paper.

\bibliographystyle{plainnat}
\bibliography{bib}
\end{document}

%% file: Commands.tex
\newtheorem{theorem}{Theorem}

\newtheorem{lemma}{Lemma}

\newcommand{\conv}{\mathrm{conv}}

\renewcommand{\phi}{\varphi}

\newcommand{\R}{\mathbb{R}}

\newcommand{\cP}{\mathcal{P}}
\newcommand{\cT}{\mathcal{T}}

\def\ds1{\mathds{1}}
\renewcommand{\epsilon}{\varepsilon}

\newlength{\minipagewidth}
\newcommand{\beq}{\begin{equation}}
\newcommand{\eeq}{\end{equation}}

\newcommand{\beqa}{\begin{eqnarray}}
\newcommand{\eeqa}{\end{eqnarray}}

\newcommand{\beqan}{\begin{eqnarray*}}
\newcommand{\eeqan}{\end{eqnarray*}}

\def\ba#1\ea{\begin{align*}#1\end{align*}} 
\def\banum#1\eanum{\begin{align}#1\end{align}} 

%% file: BL14.bbl
\begin{thebibliography}{3}
\providecommand{\natexlab}[1]{#1}
\providecommand{\url}[1]{\texttt{#1}}
\expandafter\ifx\csname urlstyle\endcsname\relax
  \providecommand{\doi}[1]{doi: #1}\else
  \providecommand{\doi}{doi: \begingroup \urlstyle{rm}\Url}\fi

\bibitem[Hatami and Norine(2011)]{HN11}
Hamed Hatami and Serguei Norine.
\newblock Undecidability of linear inequalities in graph homomorphism
  densities.
\newblock \emph{Journal of the American Mathematical Society}, 24\penalty0
  (2):\penalty0 547--565, 2011.

\bibitem[Huang et~al.(2012)Huang, Linial, Naves, Peled, and Sudakov]{HLNPS12}
Hao Huang, Nati Linial, Humberto Naves, Yuval Peled, and Benny Sudakov.
\newblock On the 3-local profiles of graphs.
\newblock \emph{arXiv preprint arXiv:1211.3106}, 2012.

\bibitem[Razborov(2007)]{Raz07}
Alexander Razborov.
\newblock Flag algebras.
\newblock \emph{Journal of Symbolic Logic}, pages 1239--1282, 2007.

\end{thebibliography}
